\documentclass{amsart}
\usepackage{graphics}
\usepackage{amsfonts,amsmath}

\begin{document}

%
%
% THEOREM Environments (Examples)-----------------------------------------
%
 \newtheorem{thm}{Theorem}[section]
 \newtheorem{cor}[thm]{Corollary}
 \newtheorem{lem}[thm]{Lemma}{\rm}
 \newtheorem{prop}[thm]{Proposition}
 \newtheorem{defn}[thm]{Definition}{\rm}
 \newtheorem{assumption}[thm]{Assumption}
 \newtheorem{rem}[thm]{Remark}
 \newtheorem{ex}{Example}
\numberwithin{equation}{section}
\def\smu{{\rm supp}\,\mu}
\def\snu{{\rm supp}\,\nu_i}
\def\e{{\rm e}}
\def\x{\mathbf{x}}
\def\bsmalpha{{\small \mathbf{\alpha}}}
\def\btheta{\mathbf{\theta}}
\def\balpha{\mathbf{\alpha}}
\def\by{\mathbf{y}}
\def\bz{\mathbf{z}}
\def\F{\mathbf{F}}
\def\R{\mathbb{R}}
\def\T{\mathbf{T}}
\def\O{\mathbf{O}}
\def\U{\mathbf{U}}
\def\N{\mathbb{N}}
\def\K{\mathbf{K}}
\def\Q{\mathbf{Q}}
\def\M{\mathbf{M}}
\def\O{\mathbf{O}}
\def\bxi{\mathbf{\xi}}
\def\C{\mathbb{C}}
\def\P{\mathbf{P}}
\def\Z{\mathbf{Z}}
\def\H{\mathbf{H}}
\def\A{\mathbf{A}}
\def\V{\mathbf{V}}
\def\AA{\overline{\mathbf{A}}}
\def\B{\mathbf{B}}
\def\T{\mathbf{T}}
\def\c{\mathbf{C}}
\def\L{\mathbf{L}}
\def\bS{\mathbf{S}}
\def\I{\mathbf{I}}
\def\Y{\mathbf{Y}}
\def\TT{\mathcal{T}}
\def\X{\mathbf{X}}
\def\f{\mathbf{f}}
\def\z{\mathbf{z}}
\def\v{\mathbf{v}}
\def\y{\mathbf{y}}
\def\d{\hat{d}}
\def\ha{\hat{a}}
\def\hb{\hat{b}}
\def\bx{\mathbf{x}}
\def\y{\mathbf{y}}
\def\p{\mathbf{p}}
\def\tz{\tilde{\mathbf{z}}}
\def\g{\mathbf{g}}
\def\w{\mathbf{w}}
\def\b{\mathbf{b}}
\def\a{\mathbf{a}}
\def\h{\mathbf{h}}
\def\v{\mathbf{v}}
\def\u{\mathbf{u}}
\def\s{\mathcal{S}}
\def\cc{\mathcal{C}}
\def\co{{\rm co}\,}
\def\vol{{\rm vol}\,}
\def\om{\mathbf{\Omega}}

\title[bounding the support of a measure]{Bounding the support of a measure 
from its marginal moments}
\author{Jean B. Lasserre}
\address{LAAS-CNRS and Institute of Mathematics\\
University of Toulouse\\
LAAS, 7 avenue du Colonel Roche\\
31077 Toulouse C\'edex 4,France}
\email{lasserre@laas.fr}
\date{}

\begin{abstract}
Given all moments of the marginals of
a measure $\mu$ on $\R^n$, one provides
(a) explicit bounds on its support and (b)
a numerical scheme to compute the smallest box 
that contains the support of $\mu$. 
\end{abstract}



\maketitle

\section{Introduction}

Inverse problems in probability are ubiquitous in several important applications, and among them shape reconstruction problems. For instance, exact recovery of two-dimensional objects from finitely many moments is possible for polygons 
and so-called quadrature domains as shown in Golub et al. \cite{polygon} and Gustafsson et al. \cite{quad}, respectively. But so far, there is no inversion algorithm 
from moments for $n$-dimensional shapes.
However, more recently Cuyt et al. \cite{golub1} have shown how to approximately recover numerically an unknown density $f$ defined on a compact region of $\R^n$, from the 
only knowledge of its moments. So when $f$ is the indicator function of a compact 
set $A\subset\R^n$ one may thus 
recover the shape of $A$ with good accuracy,
based on moment information only. 
The elegant methodology developed in \cite{golub1} 
is based on multi-dimensional homogeneous Pad\'e approximants and uses a nice Pad\'e {\it slice property}, 
the analogue for the moment approach of 
the Fourier slice theorem for the Radon transform (or projection) approach;
see \cite{golub1} for an illuminating discussion.

In this paper we are interested in the following inverse problem. 
Given an {\it arbitrary} finite Borel measure $\mu$ on $\R^n$
(not necessarily having a density with respect to the Lebesgue measure), can we compute (or at least approximate) the smallest
box $\prod_{i=1}^n[a_i,b_i]\subset\R^n$ which contains the support of $\mu$ (not necessarily compact), from the only knowledge of its moments? 

{\bf Contribution.}
Obviously, as we look for a box, the problem reduces to find for each $i=1,\ldots,n$,
the smallest interval $[a_i,b_i]$ (not necessarily compact) that contains the support of the marginal $\mu_i$ of $\mu$. 
Of course, to bound $a_i$  and $b_i$,
one possibility is to
compute zeros of the polynomials $(p_d)$, $d\in\N$,
orthogonal with respect to the measure $\mu_i$. Indeed, 
for every $d$, the smallest (resp. largest) zero of $p_d$
provides an upper bound on $a_i$ 
(resp. a lower bound on $b_i$), and there is a systematic
way to compute the $p_d$'s from from the given 
moments of $\mu$; see e.g. Gautschi \cite[\S 1.2 and \S 2.1]{gautschi}.

Our contribution is to
provide a convergent numerical scheme for computing 
this smallest interval $[a_i,b_i]$, which  (i) is based on
the only knowledge of the moments of the marginals $\mu_i$, $i=1,\ldots,n$, and (ii) {\em avoids} computing orthogonal polynomials.
For each $i$, it consists of solving $2$ hierarchies (associated with each of the end points $a_i$ and $b_i$)
of so called {\it semidefinite programs}\footnote{A semidefinite program is a convex optimization problem that can be solved (up to arbitrary but fixed precision) in time polynomial in the input size of the problem, and for which efficient public softwares are available; see e.g. \cite{boyd}} in only one variable (and therefore those semidefinite programs are generalized eigenvalue problems for which even more specialized softwares exist). Importantly, we do not make any assumption on $\mu$ and in particular, 
$\mu$ may not have a density with respect to the Lebesgue measure as in the above cited works. 
In solving the two semidefinite 
programs at step $d$ of the hierarchy, one provides an 
inner approximation $[a_d,b_d]\subset [a_i,b_i]$ such that
the sequence $(a_d)$ (resp. $(b_d)$), $d\in\N$,
is monotone nonincreasing (resp. nondecreasing) and converges to $a_i$ (resp. to $b_i$) as $d\to\infty$ (with possibly
$a_i=-\infty$ and/or $b_i=+\infty$). 
Interestingly, some explicit upper (resp. lower) bounds on $a_i$ (resp. $b_i$) in terms of the moments 
of $\mu_i$ are also available.

\section{Notation and definitions}
Let $\N$ be the set of natural numbers and
denote by $\x=(x_1,\ldots,x_n)$ a vector of $\R^n$ 
whereas $x$ will denote a scalar. Let
$\R[x]$ be the ring of real univariate polynomials in 
the single variable $x$, and denote by $\R[x]_d$ the vector space of polynomials of degree at most $d$.
Let $\Sigma[x]\subset\R[x]$ be the set of polynomials that are sums of squares (s.o.s.), and $\Sigma[x]_d$ its subspace of s.o.s. polynomials of degree at most $2d$.
In the canonical basis $(x^k)$, $k=0,\ldots,d$, of 
$\R[x]_d$, a polynomial $f\in\R[x]_d$ is written
\[x\mapsto f(x)\,=\,\sum_{k=0}^df_k\,x^k,\]
for some vector $\f=(f_k)\in\R^{d+1}$.

\paragraph{{\bf Moment matrix}} 
Given a infinite sequence $\y:=(y_k)$, $k\in\N$, indexed in the canonical basis $(x^k)$ of $\R[x]$, let $\H_d(\y)\in\R^{(d+1)\times (d+1)}$ be the real Hankel matrix defined by:
\begin{equation}
\label{mx}
\H_d(\y)(i,j)\,=\,y_{i+j-2},\qquad \forall\,i,j\leq d.
\end{equation}
The matrix $\H_d(\y)$ is called the {\it moment} matrix associated with the sequence $\y$ (see e.g. Curto and Fialkow \cite{curto} and Lasserre \cite{lasserre}). 
If $\y$ has a {\it representing measure} $\mu$ (i.e., if there exists a finite Borel measure $\mu$ such that
$y_k=\int \x^k d\mu$ for every $k\in\N$)
then
\begin{equation}
\label{mom}
\langle \f,\H_d(\y)\f\rangle \,=\,\int f(x)^2\,d\mu(x)\,\geq\,0,\qquad \forall f\in
\R[x]_d,
\end{equation}
so that $\H_d(\y)\succeq0$, where for a real symmetric matrix $\A$, the notation $\A\succeq0$ (resp. $\A\succ0$) stands for $\A$ is positive semidefinite (resp. positive definite).
\vspace{0.1cm}

\paragraph{{\bf Localizing matrix}}
Similarly, given $\theta\in\R[x]_s$ with 
vector of coefficients $(\theta_k)$, let $\H_d(\theta \,\y)\in\R^{(d+1)\times (d+1)}$ be 
the real symmetric matrix defined by:
\begin{equation}
\label{localdef}
\H_d(\theta\,\y)(i,j)\,:=\,
\sum_{k=0}^s\theta_ky_{i+j+k-2},\qquad\forall\, i,j\leq d.\end{equation}
The matrix $\H_d(\theta\,\y)$ is called the {\it localizing} matrix associated with the sequence $\y$ and the polynomial $\theta$ (see again Lasserre \cite{lasserre}). Notice that the localizing matrix with respect to the constant polynomial $\theta\equiv 1$ is the moment matrix $\H_d(\y)$ in (\ref{mx}).
If $\y$ has a representing measure $\mu$ 
with support contained in the level set $\{x\in\R:\: \theta (x)\geq0\}$, then
\begin{equation}
\label{local}
\langle \f,\H_d(\theta\,\y)\f\rangle \,=\,\int f(x)^2\theta(x)\,d\mu(x)\,
\geq\,0\qquad 
\forall\,f\in\R[x]_d,
\end{equation}
so that $\H_d(\theta\,\y)\succeq0$.\\

Finally, for a finite Borel measure $\mu$, 
denote its support by $\smu$, that is, $\smu$ is the smallest closed set $B$ such that $\mu(B^c)=0$
(where $B^c$ denotes the complement of $B$).

\section{Main result}

We may and will restrict to the one-dimensional case because
if $\mu$ is a finite Borel measure on $\R^n$, and if we
look for a box $\B:=\prod_{i=1}^n[a_i,b_i]$ such that
$\smu\subseteq\B$, then we have the following result.
For every $i=1,\ldots,n$, let $\mu_i$ be the marginal of $\mu$ with respect to the variable $x_i$.

\begin{lem}
\label{lem1}
Let $\B\subset\R^n$ be the box $\prod_{i=1}^n[a_i,b_i]$ with possibly $a_i=-\infty$ and/or $b_i=+\infty$.
Then $\smu\subseteq\B$ if and only if
$\smu_i\subseteq [a_i,b_i]$ for every $i=1,\ldots,n$.
\end{lem}
\begin{proof}
For every $i=1,\ldots,n$, let $A_i\subset\R^n$
be the Borel set $\{\x\in\R^n\,:\, x_i\in \R\setminus [a_i,b_i]\}$.
Then $0=\mu(A_i)=\mu_i(\R\setminus [a_i,b_i])$, shows that $\smu_i\subseteq [a_i,b_i]$, $i=1,\ldots,n$. Conversely, if
$\mu_i\subseteq [a_i,b_i]$, $i=1,\ldots,n$, then
$0=\mu_i(\R\setminus [a_i,b_i])=\mu(A_i)$, $i=1,\ldots,n$. But
since $\B^c=\cup_{i=1}^nA_i$ we also obtain
$\mu(\B^c)=0$.
\end{proof}
So Lemma \ref{lem1} tells us that it is enough to consider separate conditions for the marginals $\mu_i$, $i=1,\ldots,n$. Therefore, all we need to know is the sequence of moments
\[\y^i_k:=\int x_i^k\,d\mu(\x)\,=\,\int x^k\,d\mu_i(x),\qquad k=0,1,\ldots\]
of the marginal $\mu_i$ of $\mu$, for every $i=1,\ldots,n$. \\

Hence we now consider the one-dimensional case.
For a real number $a$, let $\theta_a\in\R[x]$ be the polynomial $x\mapsto \theta_{a}(x)=(x-a)$.
Recall that the support of a finite Borel measure $\mu$ on $\R$ (denoted $\smu$)
is the smallest closed set $B$ such that $\mu(\R\setminus B)=0$. For instance is $\mu$ is supported on 
$(a,b]\cup [c,d)\cup\{e\}$, with $a<b<c<d<e$ then
${\rm supp}\,\mu$ is the closed set $[a,b]\cup [c,d]\cup \{e\}$ and is contained in the interval $[a,e]$.

\subsection{Bounds on ${\rm supp}\,\mu$}
We first derive bounds on scalars $a$ and $b$ that
satisfy $\smu\subseteq (-\infty,a]$ and/or
$\smu\subseteq [b,+\infty)$.
\begin{prop}
\label{th1}
Let $\mu$ be a finite and non trivial Borel measure on the real line $\R$, with associated sequence of moments $\y=(y_k)$, $k\in\N$, all finite. 
Then: 

{\rm (i)} ${\rm supp}\,\mu\subseteq [a,+\infty)$ if and only if 
\begin{equation}
\label{lem1-1}
\H_d(\theta_a\,\y)\succeq0,\qquad\forall d\in\N.
\end{equation}

{\rm (ii)} ${\rm supp}\,\mu\subseteq [-\infty,b]$ if and only if 
\begin{equation}
\label{lem1-2}
\H_d(-\theta_b\,\y)\succeq0,\qquad\forall d\in\N.
\end{equation}
For each fixed $d\in\N$, the condition $\H_d(\theta_a\,\y)\succeq0$
(resp. $\H_d(-\theta_b\,\y)\succeq0$) determines a basic semi-algebraic set of the form 
$\{(a,\y)\,:\,p_{dk}(a,\y)\geq0,\,k=0,\ldots d-1\}$
(resp. $\{(a,\y)\,:\,(-1)^{d-k}p_{dk}(b,\y)\geq0,\,k=0,\ldots d-1\}$)
for some polynomials $(p_{dk})\subset\R[x,\y]$.

With $\y$ and $d$ fixed, the condition (\ref{lem1-1})
(resp. (\ref{lem1-2})) yields an upper bound $a\leq \underline{a}_d$ (resp. a lower bound $b\geq\overline{b}_d$), and the 
sequence $(\underline{a}_d)$ (resp. $(\overline{b}_d)$), $d\in\N$, is monotone nonincreasing (resp. nondecreasing).
In particular, 
\begin{eqnarray}
\label{b1}
a&\leq &\min\left[\frac{y_1}{y_0},\frac{y_1+y_3}{y_0+y_2}\right]\quad \mbox{and}\quad a\leq\min\left[\frac{y_1}{y_0},\frac{y_3}{y_2},\frac{y_1+y_3}{y_0+y_2}\right]
\mbox{ if }y_2\neq0\\
\label{b2}
b&\geq& \max\left[\frac{y_1}{y_0},\frac{y_1+y_3}{y_0+y_2}\right]\quad \mbox{and}\quad b\geq\max\left[\frac{y_1}{y_0},\frac{y_3}{y_2},\frac{y_1+y_3}{y_0+y_2}\right]
\mbox{ if }y_2\neq0,\end{eqnarray}
as well as 
\begin{eqnarray}
\label{b3}
a&\leq&\frac{y_0y_3-y_1y_2-
\sqrt{(y_0y_3-y_1y_2)^2-4(y_0y_2-y_1^2)(y_1y_3-y_2^2)}}{2(y_0y_2-y_1^2)}\\
\label{b4}
b&\geq& \frac{y_0y_3-y_1y_2+
\sqrt{(y_0y_3-y_1y_2)^2-4(y_0y_2-y_1^2)(y_1y_3-y_2^2)}}{2(y_0y_2-y_1^2)}\end{eqnarray}
if $(y_0y_3-y_1y_2)^2\geq 4(y_0y_2-y_1^2)(y_1y_3-y_2^2)$.
\end{prop}
\begin{proof}
(\ref{lem1-1}) and (\ref{lem1-2}) is well-known and can be found in e.g. Lasserre \cite[Theorem 3.2]{lasserre}.
Next, write 
the characteristic polynomial $t\mapsto c(t)$ of $\H_d(\theta_a,\y)$
in the form 
\[c(t)\:\left(=\,{\rm det}\,(tI-\H_d(\theta_a\,\y))\right)\,=\,t^d+
\sum_{k=0}^{d-1}(-1)^{d-k}p_{dk}(a,\y)\,t^k,\qquad t\in\R,\]
for some polynomials $(p_{dk})\subset\R[a,\y]$, and where 
$p_{dk}$ is of degree $k$ in the variable $a$. Then
$\H_d(a\,\y)\succeq0$, if and only if $c$ has all its roots nonnegative, which in turn, by Descarte's rule, happens if and only if $p_{dk}(a,\y)\geq0$, for every $k=0,\ldots,d-1$.
In fact $(a,\y)$ belong to the closure of
a convex connected component 
of $\{(a,\y): p_{d0}(a,\y)\,(={\rm det}\,H_d(\theta_a\,\y))>0\}$.
A similar argument applies for (\ref{lem1-2}) with now the characteristic polynomial
\[\tilde{c}(t)\,=\,{\rm det}\,(tI-\H_d(-\theta_b\,\y))\,=\,
(-1)^dc(-t)\,=\,
t^d+\sum_{k=0}^{d-1}p_{dk}(b,\y)\,t^k,\qquad t\in\R.\]
So with $\y$ and $d$ fixed, 
$a\mapsto p_{dk}(a,\y)$ is a univariate polynomial for every
$k$, and so
the conditions (\ref{lem1-1}) provide a bound of the form
$a\leq \underline{a}_d$ for some $\underline{a}_d$
since if $a$ satisfies (\ref{lem1-1})
then so does $a'\leq a$. 
Similarly, the conditions (\ref{lem1-2}) provide a bound of the form  $b\geq \overline{b}_d$ since if $b$ satisfies (\ref{lem1-2})
then so does $b'\geq b$. 
The scalar $\underline{a}_d$ (resp. $\overline{b}_d$) may be taken as the smallest (resp. largest) root of the polynomial
$x\mapsto p_{d0}(x,\y)$;
bounds in terms of the coefficients 
of $p_{d0}(x,\y)$ are available in the literature.

Finally (\ref{b1})-(\ref{b4}) are obtained with $d=1$ in which case (\ref{lem1-1}) and (\ref{lem1-2}) read:
\[\left[\begin{array}{cc}y_1-ay_0&y_2-ay_1\\y_2-ay_1&y_3-ay_2\end{array}\right]\,\succeq\,0;\qquad
\left[\begin{array}{cc}by_0-y_1&by_1-y_2\\ by_1-y_2&
by_2-y_3\end{array}\right]\,\succeq\,0.\]
\end{proof}
\vspace{0.2cm}

\subsection{Computing the smallest interval $[a,b]\supseteq\smu$}~

Theorem \ref{th1} provides bounds 
(some of them explicit) in terms of bounds on the largest (or smallest) root of some univariate polynomial whose coefficients
are polynomials in $\y$.
But one may also get numerical sharp bounds via solving 
the following sequence of semidefinite programs, indexed by $d$:
\begin{eqnarray}
\label{a11}
a_d&=&\max_{a}\,\left\{\, a\::\: \H_d(\theta_{a}\,\y)\succeq0\:\right\}\\
\label{a22}
b_d&=&\min_{b}\,\left\{\, b\::\: \H_d(\theta_{b}\,\y)\succeq0\:\right\},
\end{eqnarray}
where $\H_d(\theta_*\,\y)$ is the localizing matrix associated with
the polynomial $\theta_*\in\R[x]$. Observe that 
(\ref{a11}) and (\ref{a22}) are semidefinite programs with only one variable! For instance for $d=1$, (\ref{a11}) reads
\[a_1=\displaystyle\max_{a}\,\left\{\,a\::\:
\left[\begin{array}{cc}
y_1-ay_0& y_2-ay_1\\
y_2-ay_1&y_3-ay_2\end{array}\right]\succeq0\,\right\},\]
whereas (\ref{a22}) reads
\[b_1=\displaystyle\min_{b}\,\left\{\,b\::\:
\left[\begin{array}{cc}
by_0-y_1& by_1-y_2\\
by_1-y_2&by_2-y_3\end{array}\right]\,\succeq0\,\right\}\]
For more details on semidefinite programming the interested reader is referred to e.g. \cite{boyd}.
And we obtain:
\begin{thm}
\label{th2}
Let $\mu$ be a finite Borel measure with all 
moments $\y=(y_k)$ finite.  Then ${\rm supp}\,\mu\subseteq [a^*,b^*]$, with possibly $a^*=-\infty$ and/or $b^*=+\infty$, and where:

(i) $a_d$ is an optimal solution of (\ref{a11}) 
for all $d\in\N$, and the sequence $(a_d)$, $d\in\N$, is monotone nonincreasing with $a_d\downarrow a^*$ as $d\to\infty$.

(ii) $b_d$ is an optimal solution of (\ref{a22}) 
for all $d\in\N$, and and the sequence $(b_d)$, $d\in\N$, is monotone nondecreasing with $b_d\uparrow a^*$ as 
$d\to\infty$.

(iii) $a^*$ (resp. $b^*$) is the largest (resp. smallest)
scalar such that $\smu\subseteq [a^*,b^*]$.
\end{thm}
\begin{proof}
We prove the statements for (i) and (iii)
only because similar arguments hold for (ii).
We first prove that (\ref{a11}) has always a feasible solution.
If $\smu\subset [a,+\infty)$ for some $a>-\infty$,
then $a$ is obviously feasible for the semidefinite program
(\ref{a11}), for every $d\in\N$. If
there is no such $a$, consider the finite sequence of moments
$\y_d=(y_0,\ldots,y_{2d+1})$. By Tchakaloff's theorem 
(see e.g. \cite{bayer,putinar}, \cite[Theorem B.12]{lasserre}) there exists a measure $\nu$ supported on 
finitely many points $x_0\leq x_1\leq\cdots \leq x_{t}$, with $t\leq 2d+2$ (hence 
${\rm supp}\,\nu=\cup_{i=0}^{t}\{x_i\}\subset [x_0,+\infty)$), and with same moments as $\mu$, up to degree $2d+1$. Hence in view of what precedes,
$x_0$ is feasible for (\ref{a11}). Next, as every feasible solution
is bounded above by $y_1/y_0$ and as we maximize, 
it follows that (\ref{a11}) has an optimal solution $a_d$
for every $d\in\N$.

Next, observe that $a_d\leq a_k$ 
whenever $d\geq k$ because the feasible set of 
(\ref{a11}) for $d$ is contained in that for $k$ and every feasible solution is bounded above by $y_1/y_0$.
Therefore the sequence $(a_d)$, $d\in\N$, is monotone nonincreasing and thus, converges to $a^*$ with possibly $a^*=-\infty$. 

If $a^*=-\infty$ then there is no $a$ such that 
$\smu\subseteq [a,+\infty)$ because we would have 
$a_d\geq a$ for all $d$. Next, consider the case
 $a^*>-\infty$, and let $d\in\N$ be fixed. Using $\H_d(\theta_{a_d}\,\y)\succeq0$ and
the continuity of $a\mapsto \H_d(\theta_a\,\y)$, 
one obtains $\H_d(\theta_{a^*}\,\y)\succeq0$.
As $d$ fixed was arbitrary, we then obtain
$\H_d(\theta_{a^*}\,\y)\succeq0$ for every $d\in\N$. 
But then by \cite[Theorem 3.2]{lasserre}, $\mu$ is supported on the set $\{x\,:\,\theta_{a^*}(x)\geq0\}$, which shows that 
${\rm supp}\,\mu\subseteq [a^*,+\infty)$. 

Concerning (iii), if $a^*>-\infty$ then $a^*$ is the largest $a$ such that ${\rm supp}\,\mu\subseteq [a,+\infty)$ because if ${\rm supp}\,\mu\subseteq [a,+\infty)$ 
then $a$ is feasible for (\ref{a11}), for every $d\in\N$; therefore,
$a\leq a_d$ for every $d$, which in turn implies $a\leq a^*$.
\end{proof}

Next, in the case where  $\mu$ is known to have
compact support one may even consider the following single hierarchy of semidefinite programs
\begin{equation}
\label{a1}
\rho_d=\min_{b,a}\,\left\{\,b-a\::\: \H_d(\theta_{a}\,\y),\:\H_d(-\theta_b\,\y)\succeq0\:\right\},
\end{equation}
indexed by $d$, and with now two variables $a$ and $b$. We obtain the following result of which the proof is omitted.

\begin{cor}
\label{cor1}
Assume that $\mu$ has compact support. Then:

{\rm (a)} The semidefinite program (\ref{a1}) has an optimal solution $(a_d,b_d)$ for every $d\in\N$.

{\rm (b)} Let $(a_d,b_d)$, $d\in\N$, be a sequence of optimal solutions of (\ref{a1}).
As $d\to\infty$, $(a_d,b_d)\to (a^*,b^*)$ and ${\rm supp}\,\mu\subseteq [a^*,b^*]$.
Moreover, $[a^*,b^*]$ is the smallest interval which contains
${\rm supp}\,\mu$ and if the support of $\mu$ is an interval then ${\rm supp}\,\mu=[a^*,b^*]$.
\end{cor}
\vspace{0.2cm}

\subsection{Duality}
We interpret the dual of the semidefinite program (\ref{a11}). Let $\mathcal{S}_d\subset\R^{(d+1)\times (d+1)}$ be the cone of real symmetric matrices.
The dual of (\ref{a11}) is the semidefinite program:
\begin{equation}
\label{a2}
\begin{array}{rl}
a^*_d=\displaystyle\min_{\Z\in\mathcal{S}_d}&\langle \H_d( x\,\y),\Z\rangle\\
\mbox{s.t.}&\langle \H_d(\y),\Z\rangle\,=\,1\,;\:\Z\succeq0.\end{array}\end{equation}
If $\Z\succeq0$ is feasible, using the singular value decomposition
of $\Z$ one may write $\Z=\sum_i\f_i\f^T_i$ for some vectors
$(\f_i)\subset\R^{d+1}$, and so
\[1=\langle \H_d(\y),\Z\rangle\,=\,\sum_i\langle \f_i,\H_d(\y)\f_i\rangle=\sum_i\int f_i^2\,d\mu\,=\,
\int\sigma\,d\mu,\]
with $\sigma=\sum_if_i^2\in\Sigma[x]_d$, and similarly
\[\langle \H_d(x\,\y),\Z\rangle\,=\,\sum_i\langle \f_i,\H_d(x\,\y)\f_i\rangle=\sum_i\int xf_i(x)^2d\mu(x)\,=\,
\int x\sigma(x) d\mu(x).\]
Therefore, equivalently, (\ref{a2}) reads
\begin{equation}
\label{a3}
a^*_d=\displaystyle\min_{\sigma\in\Sigma[x]_d}\,\left\{\displaystyle\int x\,\underbrace{\sigma(x)d\mu(x)}_{d\nu(x)}
\::\:\int \,\underbrace{\sigma\,d\mu}_{d\nu}\,=\,1\,\right\},
\end{equation}
and $a^*_d\geq a_d$ (which is called {\it weak duality}). Indeed, 
for any two feasible solutions $a,\Z$ of (\ref{a1}) and 
(\ref{a2}) respectively, using
$\Z\succeq0$ and $\H_d(\theta_a\,\y)\succeq0$, yields
\[0\leq\,\langle \Z,\H_d(\theta_a\,\y)\rangle\,=\,\int (x-a)\,\sigma(x)\,d\mu(x)\,=\,\int x\sigma(x)\,d\mu(x)-a,\]
that is, $a\leq\int x\sigma(x)d\mu(x)$.
So in the dual semidefinite program (\ref{a3}), one searches for a sum of squares polynomial $\sigma\in\Sigma[x]_d$ of degree at most $2d$ (normalized to satisfy $\int \sigma d\mu=1$), which minimizes $\int x\sigma d\mu$. Equivalently, one searches for a probability 
measure $\nu$ with density $\sigma\in\Sigma[x]_d$ with respect to $\mu$, which minimizes the upper bound $\int xd\nu$ on the global minimum of $x$ on the support of $\mu$.

Similarly, the dual of (\ref{a22}) is the semidefinite program:
\begin{equation}
\label{a5}
b^*_d=\displaystyle\max_{\sigma\in\Sigma[x]_d}\,\left\{\displaystyle\int x\,\underbrace{\sigma(x)d\mu(x)}_{d\nu(x)}
\::\:\int \underbrace{\sigma\,d\mu}_{d\nu}\,=\,1.\,\right\}.
\end{equation}
Again, by weak duality, $b^*_d\leq b_d$ and in (\ref{a5})
one searches for 
%a sum of squares polynomial $\sigma\in\Sigma[x]_d$ of degree at most $2d$ (normalized to satisfy $\int \sigma d\mu=1$), which maximizes $\int x\sigma d\mu$.
a probability measure $\nu$ with density $\sigma\in\Sigma[x]_d$ with respect to $\mu$, which maximizes the lower bound $\int xd\nu$ on the global maximum of $x$ on the support of $\mu$.

\begin{thm}
\label{th-dual}
Suppose that $\mu$ is such that
$\H_d(\y)\succ0$ for all $d$ (for instance
if $\mu$ has no atom). Then there is no duality gap between
(\ref{a11}) and (\ref{a3}) (resp. (\ref{a22}) and (\ref{a5})),
i.e. $a_d=a^*_d$ (resp. $b_d=b^*_d$). In addition
(\ref{a3}) (resp. (\ref{a5})) has an optimal solution
$\sigma^*\in\Sigma[x]_d$ (resp. $\psi^*\in\Sigma[x]_d$), and
\begin{equation}
\label{thdual-1}
\int (x-a_d)\,\sigma^*(x)\,d\mu(x)\,=\,0\,=\,
\int (b_d-x)\,\psi^*(x)\,d\mu(x).
\end{equation}
\end{thm}
\begin{proof}
From Theorem \ref{th2}, we know that (\ref{a11}) has an optimal solution $a_d$. By Tchakaloff's theorem,
let $\nu$ be the measure supported on the finitely many points $x_0,\ldots,x_{t}$ (with $t\leq 2d+2$), and with same moments as $\mu$, up to degree $2d+1$. There are positive weights $\lambda_k$, $k=1,\ldots,t$, such that 
\[\int pd\mu=\int pd\nu=\sum_{k=0}^{t}\lambda_kp(x_k)
\qquad\forall \,p\in\R[x]_{2d+1}\]
(see e.g. \cite{bayer,putinar}).
Hence $\H_d(\theta_a\,\y)\succ0$ for every $a<x_0$,
because for every $\f\,(\neq0)\in\R^{d+1}$ (hence every $f\,(\neq0)\in\R[x]_d$), 
$\langle \f,\M_d(\y)\f\rangle=\sum_k\lambda_kf(x_k)^2>0$, and so
\begin{eqnarray*}
\langle \f,\H_d(\theta_a\,\y)\f\rangle&=&
\int f(x)^2(x-a)\,d\mu(x)\\
&=&\int f(x)^2(x-a)\,d\nu(x)\,=\,\sum_{k=0}^{t}f(x_k)^2(x_k-a)\lambda_k\,>\,0.\end{eqnarray*}
Hence every $a<x_0$ is strictly feasible for (\ref{a11}), that is,
Slater's condition\footnote{For a convex optimization 
problem $\min_\x\{\, f(\x)\,:\, g_j(\x)\geq0,\,j=1,\ldots,m\}$, Slater's condition states that there exists $\x_0$ such that
$g_j(\x_0)>0$ for every $j=1,\ldots,m$.} holds. But 
this implies that there is no duality gap, i.e.,
$a_d=a^*_d$, and in addition, the  dual (\ref{a3}) has an optimal solution $\sigma^*$; see e.g. \cite{boyd}. 
For same reasons, $b_d=b^*_d$ and (\ref{a5}) has an optimal solution. Therefore,
\[0\,=\,\int x\sigma^*(x)d\mu(x)-a^*_d\,=\,
\int (x-a^*_d)\sigma^*(x)d\mu(x)\,=\,
\int (x-a_d)\sigma^*(x)d\mu(x),\]
and similarly,
\[0\,=\,b^*_d-\int x\psi^*(x)d\mu(x)\,=\,
\int (b^*_d-x)\psi^*(x)d\mu(x)\,=\,
\int (b_d-x)\psi^*(x)d\mu(x),\]
which is the desired result (\ref{thdual-1}).
\end{proof}
In Theorem \ref{th-dual},
write $\sigma^*\in\Sigma[x]_d$ as $\sigma^*=\sum_\ell p_\ell^2$
for some polynomials $(p_\ell)\subset\R[x]_d$, with
respective coefiicient vectors $\p_d\in\R^{d+1}$.
Then by (\ref{thdual-1})
\[\int \sigma^*(x)(x-a_d)d\mu(x)\,=\,\sum_\ell\,\int p_\ell(x)^2(x-a_d)d\mu(x)\,=\,
\sum_\ell\,\langle \p_\ell,\H_d(\theta_{a_d}\,\y)\p_\ell\rangle=0,\]
so that for every $\ell$, $\langle\p_\ell,\H_d(\theta_{a_d}\,\y)\p_\ell\rangle=0$ (since $\H_d(\theta_{a_d}\,\y)\succeq0$), that is, every $\p_\ell$ is in the kernel of $\H_d(\theta_{a_d}\,\y)$.

\end{document}